\documentclass{article}
\usepackage{amsmath,amssymb,amsthm}
\usepackage{mathtools}
\usepackage{dsfont}
\usepackage{bm}
\usepackage{braket}
\usepackage{graphicx}
\usepackage[shortlabels]{enumitem}
\usepackage{hyperref}

\graphicspath{{Figs/}}

\newtheorem{thm}{Theorem}
\newtheorem{lem}[thm]{Lemma}
\newtheorem{prop}[thm]{Proposition}
\newtheorem{cor}[thm]{Corollary}
\theoremstyle{definition}
\newtheorem{defn}[thm]{Definition}

\def\ve#1{\mathchoice{\mbox{\boldmath$\displaystyle\bf#1$}}
{\mbox{\boldmath$\textstyle\bf#1$}}
{\mbox{\boldmath$\scriptstyle\bf#1$}}
{\mbox{\boldmath$\scriptscriptstyle\bf#1$}}}

\newcommand\veb{{\ve b}}
\newcommand\vecc{{\ve c}}
\newcommand\ved{{\ve d}}

\newcommand\ves{{\ve s}}

\newcommand\vev{{\ve v}}
\newcommand\vew{{\ve w}}
\newcommand\vex{{\ve x}}
\newcommand\vey{{\ve y}}
\newcommand\vez{{\ve z}}
\newcommand\veo{{\ve 0}}

\newcommand\lp{(\ref{lp})}
\newcommand\lpi{(\ref{lpi})}

\DeclareMathOperator{\rank}{rank}

\newcommand{\R}{\mathds{R}}

\newcommand{\Z}{\mathds{Z}}
\newcommand{\T}{{\intercal}}
\newcommand{\mc}{\mathcal{C}}
\newcommand{\mcb}{\mathcal{C}(B)}

\newcommand{\pol}{\{\vex\in\R^n:A\vex = \veb,\, D\vex \leq \ved\}}

\newcommand{\name}{Shifting Shadow }
\usepackage[margin=1.2in]{geometry}
\begin{document}

\title{Escaping Degeneracy by Following Shadow Edges}

\author{
Alexander E. Black\thanks{\texttt{a.black@bowdoin.edu}; Bowdoin College, Brunswick, ME, USA}
\and Sean Kafer\thanks{\texttt{spkafer@ilstu.edu}; Illinois State University, Department of Mathematics}
\and Laura Sanit\`a\thanks{\texttt{laura.sanita@unibocconi.it}; Bocconi University, Milan, Italy}
}

\maketitle
\begin{abstract}
The Simplex method is among the most widely used approaches for solving linear programs. Starting at a vertex solution of the feasible region, the algorithm proceeds through a sequence of basis exchanges (known as \emph{pivots}), each corresponding to a move along an improving edge of the polyhedron toward a better vertex.
A central challenge affecting the efficiency of the Simplex method is \emph{degeneracy}, which can cause long sequences of pivot operations that fail to change the current vertex solution.

In this paper, we prove the existence of a pivot rule which is able to escape degeneracy and follow any given \emph{shadow edge-direction} when initialized at some compatible basis, with a linear number of degenerate Simplex pivots. As a byproduct of our result, we obtain an improved bound on the number of degenerate Simplex pivots needed to solve linear programs defined on 0/1 polytopes.
\end{abstract}

\section{Introduction}\label{sec:intro}
Linear optimization, or linear programming, is a fundamental mathematical problem that concerns maximizing a linear objective function subject to a finite set of linear constraints. It not only serves as a powerful tool for decision-making in practical situations, but also provides the foundation for more advanced optimization techniques, such as nonlinear programming \cite{MINLP} and integer programming \cite{Schrijver1986}. For this reason, algorithms for linear optimization are essential and form the core of the most sophisticated commercial solvers used in industry.

From an algorithmic perspective, linear programs (LPs) may be solved in polynomial time via the ellipsoid method~\cite{Kach79} and interior point methods~\cite{Kar84}. However, such algorithms run in so-called \emph{weakly} polynomial time, that is, the number of elementary operations they perform depends on the size of the coefficients involved in the description of the LP. Finding a \emph{strongly} polynomial time algorithm for linear programming is a `holy grail' open problem in discrete mathematics and theoretical computer science, listed indeed as one of the key mathematical problems for the next century by Smale~\cite{Smale00}.

One candidate algorithm which could potentially serve in accomplishing this is the Simplex method \cite{DantzigSimplex}. The Simplex method is one of the most popular algorithms used in practice. However, its running time is poorly understood, and it has been constantly studied for more than 70 years (See \cite{terlaky+zhang,ShamirSurvey} for surveys). Given an LP, the Simplex method starts at an initial \emph{feasible basis}: this is an inclusion-wise minimal subset of tight constraints that define an extreme point solution $\mathbf{x}$ of the feasible region, which is a polyhedron. In each iteration, it replaces one tight constraint with another one: this is an operation called \emph{pivoting}. Geometrically, pivoting identifies a direction $\mathbf{u}$ along which one can move in from $\mathbf{x}$ to a point $(\mathbf{x} + \varepsilon \mathbf{u})$ with a better objective function value. The value of $\varepsilon$ is taken as the maximum possible value such that $(\mathbf{x} + \varepsilon \mathbf{u})$ is still a feasible solution. Whenever $\varepsilon >0$, the new point is a different extreme point solution of the polyhedron, and the pivoting operation is called \emph{non-degenerate}. However, it might be possible that $\varepsilon =0$, in which case the basis changes while the extreme point $\mathbf{x}$ stays the same: such pivoting operations are called \emph{degenerate}.  Although degenerate pivots do not change the objective function value, pivot directions chosen for degenerate pivots are still increasing with respect to the objective function.  This is due to the fact that Simplex does not know in advance whether its chosen pivot direction will result in a degenerate or non-degenerate pivot.

The sequence of non-degenerate pivots traces a path from an initial extreme point of the feasible region to an optimal one, on the so-called \emph{1-skeleton} of the underlying polyhedron $P$, i.e., the graph whose vertices are the extreme points of $P$, and whose edges are the 1-dimensional faces of $P$. This path is monotone with respect to the given objective function to be optimized, as each distinct extreme point along this path has a better objective function value than the previous one.

To prove a polynomial running time for the Simplex method, one needs to find a rule for exchanging tight inequalities (i.e., a pivot rule) whose total number of pivoting operations is polynomial and which is, itself, computable in polynomial time. Several pivot rules have been defined in the literature in the past 70+ years, but so far none of them have been proven to achieve this goal. In fact, it is not even known whether\textemdash independent of any pivot rule\textemdash a path of polynomial length between any two extreme points of a polyhedron is guaranteed to exist. This is currently referred to as the polynomial Hirsch conjecture~\cite{d4b74918-1611-3951-bc4e-72806d40b7e3}, and it is a major open question in discrete mathematics.

There are several classes of polyhedra for which the polynomial Hirsch conjecture is known to hold, and hence a short path between any pair of extreme points is guaranteed to exist. However, this does not directly translate into a polynomial pivot rule for LPs defined on such classes, as
even if one can compute such a path in polynomial time, degeneracy remains a crucial issue.
In general, to prove a polynomial running time for the Simplex method one would like to (i) prove the existence of a polynomial-length path on the 1-skeleton of the polyhedron (which is computable in polynomial time), and (ii) design a pivot rule able to \emph{follow such a path}, accounting for the possibility of having degenerate pivots. For general polyhedra, one can often assume that all pivots are non-degenerate by using standard perturbation techniques. However, as mentioned, for general polyhedra (i) is a major open problem. For some classes of polyhedra, if one would like to exploit the fact that (i) holds, then degeneracy cannot be ignored, as perturbation may break this property.

A recent result in \cite{KS24} argued that if one knows any improving feasible direction, then one can design a pivot rule that escapes degeneracy in a linear number of pivots. However, their pivot rule is not suited for bounding the number of \textit{non-degenerate} steps. With their strategy, one has no control over which path the non-degenerate pivot steps will follow. The authors indeed ask whether it is possible to identify conditions to maintain control of the non-degenerate Simplex pivot that is performed after the sequence of degenerate ones.

\noindent
\emph{Our results.}
In this work, we tackle the above issue for \emph{shadow pivot rules} (see Section~\ref{sec:preliminaries} for a formal definition).
Namely, we show that, for any shadow pivot rule, we can always reach the desired non-degenerate pivot direction in a linear number of degenerate pivots (\autoref{thm:make_shadow_edge_available} in Section~\ref{sec:shadows}) whenever the basis satisfies some natural conditions.
Note that every edge-direction \emph{is} a shadow edge-direction for some shadow. Hence, our results provide a \emph{general sufficient criterion} for being able to move along any improving edge via a sequence of linearly many degenerate Simplex pivots. Furthermore, we guarantee that given knowledge of this edge, this sequence of pivots may be computed in strongly polynomial time.
This result can be seen as a way to strengthen the approach in \cite{KS24}: it shows that one can employ a \textit{particular} feasible direction
to `limit' the number of degenerate pivots \emph{and at the same time} have control over which particular non-degenerate pivot one ends up performing.

The relevance of the above result becomes more prominent in the context of $0/1$-LPs -- LPs defined on polytopes whose vertices have $\{0,1\}$-coordinates.
In particular, in \cite{BDLKS24} the authors exhibit three pivot rules for $0/1$-LPs which guarantee that the number of non-degenerate steps is polynomial, and that the number of degenerate steps is finite. Each of these rules is either a shadow pivot rule or can be interpreted in the framework of shadow pivot rules. The unsolved issue in that work is opposite to the one of~\cite{KS24}: their pivot rules are able to follow the desired path on the 1-skeleton, but they do not know how to bound the number of degenerate pivots, leaving this as an open question. Combining our findings with their method, we answer this question for one of their pivot rules: we show how to modify it in such a way that \textit{both} the number of non-degenerate pivots is polynomial, and at each vertex the number of performed degenerate pivots is at most linear
(Theorem~\ref{thm:polynomial_total_pivots} in Section~\ref{sec:0/1}).

\section{Preliminaries}
\label{sec:preliminaries}

\paragraph{Notation and setting} We consider LPs of the form
\begin{equation}\label{lp}
\max\set{\,\bm c^\T\bm x\ :\, A \bm x=\bm b,\, D \bm x\leq \bm d,\, \bm x\in\R^n\,},
\end{equation}
where $A\in\Z^{m_A\times n}$, $D\in\Z^{m_D\times n}$, $\vecc\in \Z^n$, $\veb\in\Z^{m_A}$, and $\ved\in\Z^{m_D}$. Here, the matrix $A$ is allowed to be empty.
In other words, we study the LP $\max \set{\bm c^\T \bm x : \bm x \in P}$ where $P = \pol$ and the feasible region is a polytope of dimension $d$ (note that $d$ may be less than $n$).
The only simplifying assumptions we will make on $P$ are:
(i) $A$ has full rank, which can be made without loss of generality because redundant equalities are easy to identify, and
(ii) the objective function $\vecc$ is generic, which can be made without loss of generality by, e.g., perturbing $\vecc$ by some small noise (in fact, this is not strictly necessary but is done to simplify the analysis).

For a positive integer $a$, we let $[a]$ denote the set $\set{1,\ldots,a}$.
Given a vector $\bm z$, we use $\bm z(i)$ to denote the $i^\text{th}$ component of $\bm z$
and $\bm z(X)$ to denote the restriction of $\bm z$ to the components indexed by $X$.
Given a matrix $Z$, we use $Z_X$ to denote the restriction of $Z$ to the rows indexed by $X$.
Slightly abusing notation, we also use $Z_i$ to refer to the $i^{\text{th}}$ row of $Z$, considered as a (column) vector.

\paragraph{Bases, cones, pivots}

Given a vertex $\vex'$ of $P$, we let $T(P,\vex')$ denote the set of indices for inequalities tight at $\vex'$, that is
$T(P,\vex')= \set{i\in[m_D]: D_i^\T\vex' = \ved(i)}$.
We define the \emph{feasible cone} at $\vex'$ to be the set of all directions $\vez$ such that ${\vex'}+\varepsilon\vez\in P$ for some $\varepsilon >0$. That is, it is the set
\[
\mc(P,\vex') = \set{\vez \in\R^n: \,A\vez=\ve 0,\, D_{T(\vex')}\vez\leq \ve 0}
\]
When $P$ is clear from context, we may just write $T(\vex')$ or $\mc(\vex')$.

The extreme rays of $\mc(\vex')$ correspond to the \emph{edge-directions} at $\vex'$.
Let $\vez^1,\ldots,\vez^t$ be generators of the extreme rays of $\mc(\vex')$.
Then for each $i\in[t]$, there exists a unique vertex $\vex^{i}$ of $P$ which is adjacent to $\vex'$ in $P$ such that $\vez^{i} = \alpha(\vex^{i} - \vex')$ for some positive scalar $\alpha$. In this circumstance, we say that any positive scalar multiple of $(\vex^{i} - \vex)$ is an edge-direction at $\vex'$.

A basis $B$ of (\ref{lp}) is a set of $d'\coloneqq n-\rank(A)$ row-indices of $D$ such that $\binom{A}{D_B}$ is rank $n$. A basis $B$ is a \textit{feasible} basis if the unique element $\vex^B$ of $\set{\vex\in\R^n: A\vex = \veb, D_B\vex = \ved(B)}$ is in $P$. In this work, we do not require consideration of infeasible bases, so from here on, when we say ``basis," we mean ``feasible basis." For any basis $B$, $\vex^B$ is a vertex of $P$, and we call it the vertex of $P$ associated to $B$.
This definition of the term basis---as a collection of inequality indices rather than variable indices---contrasts with the way it is usually introduced in a typical description of the Simplex method. However, these two perspectives are completely equivalent (see e.g., \cite{BDLKS24}).

We define the \emph{basic cone} associated to a basis $B$ as
\[
\mc(P,B) = \set{ \bm z \in \R^n : \, A\bm z = \bm 0,\, D_B \bm z \leq \bm 0}.
\]
Again, when $P$ is clear from context, we may just write $\mcb$.
The extreme rays of $\mcb$ correspond to the available \textit{pivot-directions} at $B$.
Let $\vez^1,\ldots,\vez^{d'}$ be generators of the extreme rays of $\mcb$. Then for each $i\in[d']$, there exists at least one basis $B^i$ which is \textit{adjacent} to $B$ (meaning, their intersection has size $d'-1$) such that $D_{B\cap B^i}\vez^i = \veo$. In this circumstance, we say that $\vez^i$ is a \textit{pivot-direction} at $B$, and in this work we identify a choice of pivot with a choice of the corresponding pivot-direction. Note that the pivot-direction $\vez^i$ only determines which row index of $D$ will leave the basis. In the presence of degeneracy, there may be multiple choices for the row index that will enter. In particular, given a basis $B$ and a pivot-direction $\vez^i$, let $\alpha'=\max\set{\alpha:\vex^B+\alpha\vez^i\in P}$ and let $\vey = \vex^B + \alpha'\vez^i$. Then $j\in[m_D]$ is a valid choice of entering index if $j\in T(\vey)$ and $D_j^\T\vez^i > 0$. Note that $\alpha'$ may equal 0, in which case $\vey = \vex'$. In this case, we call the pivot a \textit{degenerate pivot}. If $\vey \neq \vex'$, we call the pivot a \textit{non-degenerate pivot}.

Although there may be exponentially many edge-directions at a vertex $\vex'$ of a polyhedron, there are only $d'\leq n$ pivot-directions at B. Thus, one can always efficiently enumerate the set of available pivot-directions at a basis $B$.
The Simplex method solves LPs by moving from a basis to an adjacent basis, typically by selecting pivot-directions which are $\vecc$-increasing. We therefore call a \textit{Simplex pivot} any choice of an adjacent basis such that the corresponding pivot-direction $\vez^i$ satisfies $\vecc^\T\vez^i > 0$.

\paragraph{Shadow pivot rule and shadow edges}
Given an LP $(P)$ of the form (\ref{lp}) with feasible region $P$ and a vector $\vev\in\R^n$, we can obtain a projection $\pi: \R^{n} \to \R^{2}$ by $\pi(\vex) = (\vev^\T \vex, \vecc^\T \vex)$.
Applying this projection to $P$ yields a polygon $\pi(P)$, which is called a \emph{shadow} of $P$.

\begin{defn}\label{defn:basis_shadow}
    Given an LP of the form (\ref{lp}) with feasible region $P$ and a vector $\vev\in\R^n$, we say that a basis $B$ (resp. a vertex $\vex'$) is on the $(\vev,\vecc)$-shadow of $P$ if
    \begin{enumerate}[1)]
        \item there exists a convex combination $\vew$ of $-\vev$ and $\vecc$ such that $\vew$ is a conic combination of the rows of $\binom{\pm A}{D_B}$ (resp. $\binom{\pm A}{D_{T(\vex')}}$), and
        \item whenever no direction $\vez$ in $\mcb$ (resp. $\mc(\vex')$) satisfies $\vev^\T\vez < 0$, then there is no $\vecc$-increasing direction $\vez$ in $\mcb$ (resp. $\mc(\vex')$) satisfying $\vev^\T\vez = 0$.
    \end{enumerate}
When $B$ (resp. $\vex'$) is on the $(\vev,\vecc)$-shadow of $P$, we say that also $\mcb$ (resp. $\mc(\vex')$) is on the $(\vev,\vecc)$-shadow of $P$.
\end{defn}

Condition 2) in the definition above is merely to rule out certain degenerate situations at an initial basis, and it is often a standard assumption made in the context of shadow pivot-rules. In particular, it implies that at a basis which is $(-\vev)$-maximal (in which case all directions $\vez$ in $\mcb$ satisfy $\vev^\T\vez \geq 0$), all $\vecc$-increasing directions in $\mcb$ are also $\vev$-increasing. In \autoref{lem:shadow_equivalences}, we will show that this holds at all other bases on the shadow, and the purpose of condition 2) above is to define it to hold at an initial basis.

We also note that if $B$ corresponds to $\vex'$, we have that $B\subseteq T(\vex')$. Likewise, if $\vew$ is a conic combination of the rows of $\binom{\pm A}{D_{T(\vex')}}$, then
it is a conic combination of the rows of $\binom{\pm A}{D_{B}}$ for some $d'$ linearly independent rows indexed by $B\subseteq T(\vex')$.
One can then see that a vertex $\vex'$ is on the $(\vev,\vecc)$-shadow if and only if there exists a basis $B$ on the $(\vev,\vecc)$-shadow such that $\vex' = \vex^B$.

\begin{defn}
    Given an LP $(P)$ of the form (\ref{lp}) and extreme-point solution $\vex'$ on the $(\vev,\vecc)$-shadow, the shadow edge at $\vex'$, denoted $\ves_{\vex'}$, is the edge-direction at $\vex'$ that maximizes $\frac{\vecc^\T\vez}{\vev^\T\vez}$ over all edge-directions $\vez$ at $\vex'$ satisfying $\vecc^\T\vez > 0$.
\end{defn}

We also let $R(P,\vex') = \set{i\in T(\vex') : D_i\ves_{\vex'} = 0}$ be the set of indices of inequalities tight at the shadow edge (and write $R(\vex')$ if $P$ is clear from context).

\begin{defn}
    Given an LP $(P)$ of the form (\ref{lp}) and a basis $B$ on the $(\vev,\vecc)$-shadow, the shadow pivot-direction at $B$ is the pivot-direction $\vez^j$ that maximizes $\frac{\vecc^\T\vez}{\vev^\T\vez}$ over all pivot-directions $\vez$ at $B$ satisfying $\vecc^\T\vez>0$.
\end{defn}

\section{Degenerate pivots can follow shadow edges quickly}\label{sec:shadows}

Here we consider arbitrary shadow pivot rules for arbitrary LPs. The main purpose of this section is to prove Theorem~\ref{thm:make_shadow_edge_available}.
Essentially, it states that, if we are at a basis $B$ on the $(\vev,\vecc)$-shadow which corresponds to the vertex $\vex'$, then in a sequence of at most $n$ Simplex pivots we can pivot along the shadow edge $\ves_{\vex'}$.  Recall that this is an improvement over the results of \cite{KS24} wherein the authors are are able to escape degeneracy in the same number of degenerate pivots, but cannot control the edge along which they perform a non-degenerate pivot.

\begin{thm}\label{thm:make_shadow_edge_available}
    Let $(P)$ be an LP of the form (\ref{lp}) with feasible region $P$, $B$ be a basis on the $(\vev,\vecc)$-shadow, and $\vex' = \vex^B$. Then if $\ves_{\vex'}$ is given, one can compute in strongly polynomial time a sequence of degenerate Simplex pivots from $B$ to a basis $B^*$ at which $\ves_{\vex'}$ is an available pivot direction. Furthermore, this sequence of Simplex pivots is of length $|B^*\cap R(\vex')| - |B\cap R(\vex')| \leq n - \rank(A) - 1 - |B\cap R(\vex')| $.
\end{thm}

We remark that it is important that our basis $B$ is on the $(\vev,\vecc)$-shadow and
that $\ves_{\vex'}$ is known.
Before giving a proof, we need a few lemmas.

\paragraph{Part I: Properties of bases on shadows} We start with a set of lemmas reproducing some results in \cite{BDLKS24}, which we have to reinterpret in the context of shadow-pivot rules. 

Lemma~\ref{lem:shadow_equivalences} gives a convenient equivalent condition to being on the shadow.

\begin{lem}\label{lem:shadow_equivalences}[from \cite{BDLKS24}]
    Let $(P)$ be an LP of the form (\ref{lp}) with feasible region $P$ and with generic objective function $\vecc$. Let $C$ be either
    (i) a basic cone for some basis $B$ which is not $\vecc$-maximal, or
     (ii) a feasible cone for some extreme point solution $\vex'$ of $P$ which is not $\vecc$-maximal.
    Then the following are equivalent:
    \begin{enumerate}[(A)]
        \item $C$ is on the $(\vev,\vecc)$-shadow.
        \item For all $\vez\in C$ such that $\vecc^\T\vez > 0$, we have that $\vev^\T\vez > 0$.
    \end{enumerate}
\end{lem}

Next lemma establishes a nice characterization of shadow edges.

\begin{lem}\label{lem:shadow_edge_LP}[from \cite{BDLKS24}]
    Let $(P)$ be an LP of the form (\ref{lp}) and let $\vex'$ be a vertex on the $(\vev,\vecc)$-shadow. The shadow edge at $\vex'$ is an optimal solution of the LP $(Q^{\vex'})$:
    \begin{equation}
    	\max \{ \vecc^\T \vez : \vev^\T\vez = 1, \; \vez \in \mc(\vex')\} \tag{$Q^{\vex'}$}\label{eq:shadow_edge_lp}
    \end{equation}
\end{lem}

Applying Lemma~\ref{lem:shadow_edge_LP} to the LP obtained from $(P)$ by removing all inequalities indexed by $T(\vex^B)\setminus B$,
we get an analogous statement for bases as a corollary.

\begin{cor}\label{cor:shadow_pivot_LP}
    Let $(P)$ be an LP of the form (\ref{lp}) and $B$ be a basis on the $(\vev,\vecc)$-shadow. The shadow pivot direction at $B$ is an optimal solution of the LP $(Q^B)$:
        \begin{equation}
		\max \{ \vecc^\T \vez : \vev^\T\vez = 1, \; \vez\in \mc(B)\} \tag{$Q^{B}$}\label{eq:shadow_pivot_lp}
    \end{equation}
\end{cor}

The following lemma states a well-known fact about shadow pivot rules.

\begin{lem}\label{lem:stay_on_shadow}[from \cite{BDLKS24}]
Let $(P)$ be an LP of the form (\ref{lp}) and $B$ be a basis on the $(\vev,\vecc)$-shadow which is not $\vecc$-maximal. Let $B'$ be the basis obtained by performing a pivot along the shadow-pivot direction at $B$. Then $B'$ is on the $(\vev,\vecc)$-shadow.
\end{lem}

\paragraph{Part II: Escaping degeneracy}

We can now proceed with the novel contributions of this section.
An immediate consequence of \autoref{lem:shadow_edge_LP} is that any $\vecc$-increasing direction $\vez'$ with $\frac{\vecc^\T\vez'}{\vev^\T\vez'}> \frac{\vecc^\T\ves_{\vex'}}{\vev^\T\ves_{\vex'}}$ is necessarily an infeasible direction at $\vex'$. In the next key lemma, we show that in fact, any such direction $\vez'$ violates at least one of the constraints which is kept tight by $\ves_{\vex'}$.

\begin{lem}\label{lem:violate_edge_constraint}
Let $(P)$ be an LP of the form (\ref{lp}) and $\vex'$ be a vertex on the $(\vev,\vecc)$-shadow which is not a $\vecc$-maximizer. Let $\vez'$ be any vector such that $A \vez'=0$, $\vecc^\T\vez' > 0$ and $\frac{\vecc^\T\vez'}{\vev^\T\vez'} > \frac{\vecc^\T\ves_{\vex'}}{\vev^\T\ves_{\vex'}}$. Then there exists $j\in R(\vex')$ such that $D_j^\T\vez' > 0$.
\end{lem}

\begin{proof}
Let $\ves\coloneqq \ves_{\vex'}$. Suppose for the sake of contradiction that there exists $\vez'\in\R^n$ such that $\vecc^\T\vez' > 0$, $\frac{\vecc^\T\vez'}{\vev^\T\vez'} > \frac{\vecc^\T\ves}{\vev^\T\ves}$, but $D_j^\T\vez' \leq 0$ for all $j\in R(\vex')$.
Note that by \autoref{lem:shadow_equivalences},
$\vev^\T\ves > 0$, and therefore $\vev^\T\vez' > 0$.
Assume without loss of generality that $\vecc^\T\vez' = \vecc^\T\ves = 1$. We have that $D^\T_j\ves = 0$ for all $j\in R(\vex')$ and $D^\T_l \ves < 0$ for all $l\in T(\vex')\setminus R(\vex')$. Let $\vew = (1-\lambda)\vez' + \lambda \ves$ for $\lambda\in (0,1)$ chosen so that for all $l\in T(\vex')\setminus R(\vex')$, $D^\T_l\vew < 0$. Note that this holds for any $\lambda$ sufficiently close to $1$.
Then we have that $\vecc^\T\vez' = \vecc^\T\vew = \vecc^\T\ves$ and $\vev^\T\vez' < \vev^\T\vew < \vev^\T\ves$, and therefore
$
\frac{\vecc^\T\vew}{\vev^\T\vew} > \frac{\vecc^\T\ves}{\vev^\T\ves}.
$
In particular, this implies by \autoref{lem:shadow_edge_LP} that $\vew\notin \mc(\vex')$, and so $\vew$ violates at least one inequality $D_k^\T\vez \leq 0$ from the system $D_{T(\vex')}\vez\leq \veo$. As above, $k\notin T(\vex')\setminus R(\vex')$, and so we have that $k\in R(\vex')$.
On the other hand, we have by assumption that $D_k^\T\vez' \leq 0$. Then since $D_k^\T\ves= 0$ and $\vew$ is a convex combination of $\vez'$ and $\ves$, we have that $D_k^\T\vew \leq 0$, a contradiction.
\end{proof}

\autoref{lem:violate_edge_constraint} implies that, by performing shadow pivots, we may at each step add some constraint which is tight at the shadow edge, providing hope that we might escape degeneracy along the shadow edge in at most $n$ degenerate Simplex pivots. However, since it is possible for edges to be over-defined by non-redundant inequalities, in principle it could happen that a shadow pivot also removes some constraint which is tight at the shadow edge.
To deal with this, we prove \autoref{lem:add_shadow_constraint} which is inspired by the strategy of \cite{KS24}. It shows that if one performs shadow pivots on a sequence of \textit{face-restrictions} of our LP, one can guarantee that when a constraint tight at the shadow edge is added, we are allowed to \emph{never} remove it.

Formally, for an LP $(P)$ of the form (\ref{lp}) and any $I\subseteq[m_D]$, we define
the \emph{face-restricted} LP $(P_I)$ to be the LP obtained from $(P)$ by requiring the inequalities indexed by $I$ to be satisfied with equality.
That is, $(P_I):= \max\set{\vecc^\T\vex : \vex\in P_I}$, where $P_I := \set{\vex\in\R^n: A\vex = \veb, D_I\vex = \ved(I), D_{[m_D]\setminus I}\vex \leq \ved([m_D]\setminus I)}$.

With this notation, we first argue that if a basis is on the shadow, any subset of that basis is a basis on the shadow of a face of $P$.

\begin{prop}\label{prop:restricted_shadow}
    Let $(P)$ be an LP of the form (\ref{lp}), $B$ be a basis on the $(\vev,\vecc)$-shadow, and let $X\subseteq B$. Then $B\setminus X$ is on the $(\vev,\vecc)$-shadow of $P_X$.
\end{prop}
\begin{proof}
    Since $B$ is on the $(\vev,\vecc)$-shadow of $P$, by definition there exists a convex combination $\vew$ of $-\vev$ and $\vecc$ such that $\vew$ is a conic combination of the rows of $\pm A$ and $D_B$. Then $\vew$ is also a conic combination of the rows of
    $\pm A$, $\pm D_X$, and $D_{B\setminus X}$.
    Therefore, $B\setminus X$ is on the $(\vev,\vecc)$-shadow of $P_X$.
\end{proof}

We will use face restricted LPs to give us a measure of progress that we will later use to bound the number of degenerate steps in the original LP.

\begin{lem}\label{lem:add_shadow_constraint}
    Let $(P)$ be an LP of the form $(\ref{lp})$, $\vex'$ be a vertex on the $(\vev,\vecc)$-shadow, and $B$ be a basis corresponding to $\vex'$. Let $I = B\cap R(\vex')$. Assume $B\setminus I$ is on the $(\vev,\vecc)$-shadow of $P_I$, and $\ves_{\vex'}$ is not an available pivot-direction at $B$. Then there exists a Simplex pivot at $B$ such that the resulting basis $B'$ satisfies:
    \begin{enumerate}
        \item $B'$ corresponds to $\vex'$,
        \item $|B'\cap R(\vex')| = |B\cap R(\vex')| + 1$, and
        \item $B'\setminus K$ is on the $(\vev,\vecc)$-shadow of $P_{K}$ where $K = B'\cap R(\vex')$.
    \end{enumerate}
    Furthermore, such a Simplex pivot can be identified in strongly polynomial time.
\end{lem}
\begin{proof}
    Consider the LP $(P_I)$ and let $\tilde{B} = B\setminus I$. Then $\tilde B$ is a basis of $(P_I)$ which corresponds to $\vex'$, and by assumption, $\tilde B$ is on the $(\vev,\vecc)$-shadow of $P_I$. By \autoref{lem:shadow_edge_LP}, $\ves_{\vex'}$ is an optimal solution to the LP
    $\max \{ \vecc^\T \vez : \vev^\T\vez = 1, \;
		\vez\in \mc(P,\vex')\}$.

    So $\ves_{\vex'}\in\mc(P,\vex')$. Furthermore, by definition $D_i^\T\ves_{\vex'} = 0$ for all $i\in I$, so $\ves_{\vex'}\in \mc(P_I,\vex')$. Then, since $\mc(P_I,\vex')$ is a restriction of $\mc(P,\vex')$, it follows that $\ves_{\vex'}$ is an optimal solution to the LP
    $\max \{ \vecc^\T \vez : \vev^\T\vez = 1, \;\vez\in \mc(P_I,\vex')\}.$
    Let $\vez'$ be the shadow pivot direction at $\tilde B$ in $(P_I)$.
    By \autoref{cor:shadow_pivot_LP}, $\vez'$ is an optimal solution to the LP
     $\max \{ \vecc^\T \vez : \vev^\T\vez = 1, \;\vez\in \mc(P_I,\tilde B) \}.$
    Note that since $\ves_{\vex'}$ is not an available pivot-direction at $B$ it is also not an available pivot-direction at $\tilde B$. Then $\vez'\neq \ves_{\vex'}$. Nevertheless, $\ves_{\vex'}\in \mc(P_I,B)$, and so we have that
    $
    \frac{\vecc^\T\vez'}{\vev^\T\vez'} > \frac{\vecc^\T\ves_{\vex'}}{\vev^\T\ves_{\vex'}},
    $
    where the strictness of the inequality follows from the fact that $\frac{\vecc^\T\vez'}{\vev^\T\vez'} \neq \frac{\vecc^\T\ves_{\vex'}}{\vev^\T\ves_{\vex'}}$, because $\vecc$ is generic (see Lemma 8 in \cite{BDLKS24}).
 Therefore, by \autoref{lem:violate_edge_constraint} there exists $j\in R(\vex')$ such that $D_j^\T\vez' > 0$. Furthermore, by the definition of $P_I$, $D_i^\T\vez' = 0$ for all $i\in I$, so $j\in R(\vex')\setminus I$.
    Then at the basis $B$ of $(P)$, $\vez'$ is an available, degenerate pivot direction which satisfies $\vecc^\T\vez' > 0$. Since $j\in R(\vex')\subseteq T(\vex')$, we have that the inequality $D_j^\T\vex \leq \ved(j)$ is tight at $\vex'$, and since we also have that $D_j^\T\vez' > 0$, we therefore have that $j$ is a valid choice of entering inequality index.
    Let $i$ be the index removed from $B$ to pivot in direction $\vez'$.
    Let $B' = (B\setminus\set{i})\cup \set{j}$ be the basis of $(P)$ corresponding to $\vex'$ which is obtained from $B$ by selecting $\vez'$ as a pivot direction and choosing $j$ as the entering inequality index. Let $K = B'\cap R(\vex')$. Note then that $|K| = |I| + 1$. It remains to show that $B'\setminus K$ is on the $(\vev,\vecc)$-shadow of $P_K$.

    Let $\tilde{B'} = (\tilde B \setminus \set{i})\cup \set{j}$. Then $\tilde{B'}$ is the basis of $(P_I)$ obtained from the basis $\tilde B$ of $(P_I)$ by pivoting along the shadow pivot direction at $\tilde B$ in $(P_I)$. By \autoref{lem:stay_on_shadow}, $\tilde {B'}$ is on the $(\vev,\vecc)$-shadow of $P_I$. By \autoref{prop:restricted_shadow}, $\tilde{B'}\setminus \set{j}$ is on the $(\vev,\vecc)$-shadow of $P_{I\cup\set{j}}$. We have $K = I\cup \set{j}$, and $\tilde{B'} = B' \setminus(I\cup\set{j}) = B'\setminus K$. Then $B'\setminus K$ is on the $(\vev,\vecc)$-shadow of $P_K$, as desired.

We now discuss the running time needed to identify the pivot. Since the chosen pivot direction $\vez'$ is the shadow pivot direction at $\tilde B$, it is the pivot direction that maximizes $\frac{\vecc^\T\vez}{\vev^\T\vez}$ over all $\vecc$-increasing extreme ray generators of $\mc(P_I,\vex')$. Equivalently, it is the pivot direction that maximizes $\frac{\vecc^\T\vez}{\vev^\T\vez}$ over all $\vecc$-increasing extreme ray generators $\vez$ of $\mc(P,B)$ which satisfy the additional condition that $D_I^\T\vez = \veo$. This pivot direction is easily identified as long as we know $I = B\cap R(\vex')$, which itself is easy to determine if $\ves_{\vex'}$ is known. Likewise, if $\ves_{\vex'}$ is known then a choice $j$ of entering inequality index can be found by testing whether $D_j^\T\vez' > 0$ for each $j\in R(\vex')\setminus I$.
\end{proof}

We are now ready to prove \autoref{thm:make_shadow_edge_available}.

\begin{proof}[Proof of \autoref{thm:make_shadow_edge_available}]
    We can compute $T(\vex')$ by computing $D_i^\T\vex'$ for all $i\in[m_D]$, and given $\ves_{\vex'}$ we can then compute $R(\vex')$ by computing $D_i^\T\ves_{\vex'}$ for all $i\in T(\vex')$. If $\ves_{\vex'}$ is an available pivot direction at $B$, this can be easily identified and we are done. Let $I=B\setminus R(\vex')$.
    By \autoref{prop:restricted_shadow}, $B\setminus I$ is on the $(\vev,\vecc)$-shadow of $P_I$. By \autoref{lem:add_shadow_constraint}, we can identify in strongly polynomial time a Simplex pivot at $B$ such that the resulting basis $B'$ satisfies that $|B'\cap R(\vex')| = |B\cap R(\vex')| +1$. Furthermore, we either have that (1) $\ves_{\vex'}$ is an available pivot direction at $B'$, or (2) $B'$ itself satisfies the hypotheses of \autoref{lem:add_shadow_constraint}.
Therefore, after an application of \autoref{lem:add_shadow_constraint}, we increase the size of the intersection of $R(\vex')$ with our current basis, and if $\ves_{\vex'}$ is not an available pivot direction, we can apply \autoref{lem:add_shadow_constraint} again. If we do not reach a basis at which $\ves_{\vex'}$ is an available pivot direction within $n - \rank(A) -|I|-2$ pivots, after the $(n - \rank(A) - 1 -|I|)$-th application of \autoref{lem:add_shadow_constraint}, we reach a basis $B^*$ where $|B^*\cap R(\vex')| = |I| + (n - \rank(A) - 1 - |I|) = n - \rank(A) - 1$. Furthermore, since the rows of $D$ indexed by $B$ are linearly independent, the rows of $D$ indexed by $B^*\cap R(\vex')$ are linearly independent. Then $\ves_{\vex'}$ is the unique (up to scaling) vector satisfying $D_j^\T\vez = 0$ for all $j\in B^*\cap R(\vex')$. Thus, $\ves_{\vex'}$ is an available pivot direction at $B^*$, as desired.
Since the size of the intersection of $R(\vex')$ with our current basis increases by one after each pivot, the number of Simplex pivots performed is exactly $|B^*\cap R(\vex')| - |B\cap R(\vex')|$.
\end{proof}

\section{Degeneracy via shadow edges in 0/1-LPs}\label{sec:0/1}

In this section, we will apply \autoref{thm:make_shadow_edge_available} in the context of 0/1-LPs.
To simplify the exposition and the analysis, we assume we are dealing with LPs of the form
\begin{equation}\label{lpi}
\max\set{\vecc^\T\vex : A\vex \leq \veb, \vex\in\R^n}
\end{equation}

Note that this assumption can be made without loss of generality:
given an LP of the form \lp, we can easily transform it to the form \lpi\ by replacing each equality constraint with 2 inequalities. We further assume without loss of generality that $A\vex \leq \veb$ includes the inequalities $-\vex\leq\veo$, $\vex\leq \ve 1$.

We point out that leveraging \autoref{thm:make_shadow_edge_available} is not as straightforward as it may seem. In particular, the sequence of basis exchanges it describes assume that we \textit{start} at a basis on the $(\vev,\vecc)$-shadow, but by performing pivots which are shadow pivots in \textit{face-restrictions} of $(P)$ (rather than performing the shadow pivots of $(P)$ itself) the bases we visit may no longer be on the shadow of $(P)$. As such, when we do perform a non-degenerate pivot along $\ves_{\vex'}$ the basis $B'$ at which we arrive may not be on the shadow of $(P)$. In such a circumstance, we would not be capable of repeating the process outlined by \autoref{thm:make_shadow_edge_available}. Thus, leveraging this result requires some way to get back onto the shadow.
For 0/1-LPs, this can be done by exploiting the so-called \emph{coordinate basis}, introduced in~\cite{BDLKS24}.

\begin{defn}[Definition $1$ in \cite{BDLKS24}]
    \label{defn:coordinate_basis}
    Given a 0/1-LP $(P)$ of the form (\ref{lpi}) and an extreme point solution $\vex'$, let $U = \set{i\in[n] : \vex'(i) = 1}$ and $L=\set{i\in[n] : \vex'(i)=0}$. We define the coordinate basis at $\vex'$ to be the unique basis $B$ whose basic cone is
    \[
    \mcb = \set{\vez\in\R^n: \vez(U)\leq \ve 0, -\vez(L)\leq \ve 0}.
    \]
\end{defn}

That is, the coordinate basis at $\vex'$ is precisely the basis whose row-indices correspond to the coordinate bounds which are tight at $\vex'$. It is not hard to see that, for the vector $\vev=(\ve1-2\vex')$, we have that $(-\vev)$ is the sum (and therefore a conic combination) of the rows indexed by the coordinate basis. This makes $\vex'$ the unique $(\ve 1 - 2\vex')$-minimizer in the feasible region, and so gives the following.

\begin{lem}[Lemma 1 in \cite{BDLKS24}]\label{lem:coordinate_basis_on_shadow}
    Given a 0/1-LP $(P)$ of the form (\ref{lpi}) and a vertex $\vex'$, let $\vev = \ve 1 - 2\vex'$. Then, $\vex'$ is on the $(\vev,\vecc)$-shadow and the coordinate basis $B$ at $\vex'$ is on the $(\vev,\vecc)$-shadow.
\end{lem}

Our starting point is the True Steepest-Edge pivot rule, defined in \cite{BDLKS24}, which we briefly recall next.
At a vertex $\vex'$ defined by an initial basis $B$, this pivot rule performs two sequences of pivoting operations. The first is a sequence of at most $n$ degenerate pivots to arrive at the coordinate basis $B'$ at $\vex'$. After this, it performs a sequence of Simplex pivots by selecting a pivot-direction $\bm z'$ that maximizes $\frac{\bm c^\T\bm z}{\bm v^\T\bm z}$ over all pivot-directions satisfying $\bm c^\T\bm z > 0$. 
The authors prove that in a finite but undetermined number of degenerate pivots, the True Steepest-Edge pivot rule performs a non-degenerate pivot.
Moreover, they show that when the True Steepest-Edge pivot rule ultimately performs this non-degenerate pivot at $\vex'$, it is guaranteed to be along the edge-direction that maximizes $\frac{\vecc^\T\vez}{\|\vez\|_1}$ over all edge-directions at $\vex'$. This direction is called the steepest edge-direction in \cite{DKL2022}, and it is proved there that moving along this direction yields an optimal solution within a strongly polynomial number of non-degenerate steps.  They state as an open question whether one can bound the number of Simplex pivots(i.e. pivots along $\vecc$-increasing directions) necessary to pivot along the steepest edge-direction starting from the coordinate basis.  Using \autoref{thm:make_shadow_edge_available} we answer this question, showing that this can be done in a linear number of degenerate Simplex pivots.

As discussed earlier, it is valuable to restrict our attention to Simplex pivots only, even for degenerate pivots which do not actually change the objective function value.  This is due to the fact that the Simplex method is agnostic to what type of pivot it is performing---it does not know ahead of time whether a pivot will ultimately be degenerate or not.

This value persists even despite the fact that in \cite{BDLKS24}, the first sequence of pivots they perform to arrive at the coordinate basis may or may not be Simplex pivots.
Their True Steepest-Edge pivot rule only performs such pivots  when moving to the coordinate basis.  This is done only because of special properties of the coordinate basis which make it uniquely simple and computationally inexpensive to pivot to directly: This basis is known \textit{a priori} (i.e. requires no computation to find), and since its basic submatrix is an identity matrix, it is trivial to put the LP into canonical form for this basis (e.g. it is not necessary to perform a typically-expensive matrix-inverse computation).
For these reasons, it remains important in \cite{BDLKS24} (and therefore, here) that all \textit{subsequent }degenerate pivots remain Simplex pivots.

Therefore, even though our results in Section \ref{sec:shadows} require fore-knowledge of the shadow edge $\ves_\vex'$, in order to answer the open question of \cite{BDLKS24} we are precluded from pivoting in an \textit{arbitrary }way to a basis at which $\ves_\vex'$ is an available pivot (in the same way that the authors of \cite{BDLKS24} pivot directly to the coordinate basis).  Instead, we must utilize \autoref{thm:make_shadow_edge_available} to ensure that the degenerate pivots performed are indeed Simplex pivots. To this end, we modify the second portion of the True Steepest-Edge pivot rule, yielding a pivot rule which escapes degeneracy from the coordinate basis with a \emph{linear} number of Simplex pivots.

\begin{defn}
    Given a 0/1-LP as in \lpi, a basis $B$ and a vertex $\vex'$ associated to $B$, let $\vev = \ve1 - 2\vex'$. The \name pivot-rule works as follows:
    \begin{enumerate}
        \item If we have not yet visited the coordinate basis $B'$ corresponding to $\vex'$, it performs a pivot which adds an element of $B'$ and removes one not in $B'$.
        \item If $\ves_{\vex'}$ is an available pivot direction, it selects it.
        \item Otherwise, it selects a pivot direction $\vez'$ that maximizes $\frac{\vecc^\T\vez}{\vev^\T\vez}$ over all pivot directions $\vez$ satisfying (a) $\vecc^\T\vez > 0$ and (b) $A_i^\T\vez = 0$ for all $i\in B\cap R(\vex')$.
    \end{enumerate}
    Finally, it selects as an entering index any $j\in R(\vex')$ with $A_j^\T\vez' > 0$.
\end{defn}

This definition differs from that of the True Steepest-Edge pivot rule of \cite{BDLKS24} only in points 2 and 3.
It is easy to see that this choice of entering index yields indeed a new basis, i.e., that the resulting index set indexes a set of linearly independent rows of $A$.

\begin{thm}\label{thm:polynomial_total_pivots}
    Given a 0/1-LP of the form \lpi, and any initial basis $B$, the \name pivot-rule reaches an optimal solution in a number of pivots which is bounded above by a function $f(n)$ which is a polynomial that depends on $n$ only.
\end{thm}
\begin{proof}
Let $B$ be the first basis visited which corresponds to a vertex $\vex'$. First, the \name pivot-rule performs a sequence of at most $n$ degenerate pivots to arrive at the coordinate basis $B'$ at $\vex'$. By \autoref{lem:coordinate_basis_on_shadow}, the coordinate basis $B'$ is on the $(\vev,\vecc)$-shadow of $P$ for the choice of $\vev$ in the statement of \autoref{lem:coordinate_basis_on_shadow}.
    Then, the \name pivot-rule performs the sequence of Simplex pivots as described in \autoref{thm:make_shadow_edge_available}, which is of length at most $n$, until it performs a non-degenerate Simplex pivot along the shadow edge $\ves_{\vex'}$.
    As shown in \cite{BDLKS24}, for this choice of $\vev$, this edge-direction is the steepest edge-direction at $\vex'$, and thus the total number of non-degenerate pivots performed by the \name pivot-rule is bounded by a polynomial in $n$ only.
\end{proof}

As in~\cite{KS24}, our approach relies on the knowledge of the shadow edge-direction, and hence it presumes some computation of $\ves_{\vex'}$. For 0/1-LPs, (\ref{eq:shadow_edge_lp}) can be solved in \textit{strongly} polynomial time \cite{GLS}, but for general LPs this computation is not any easier than solving the LP itself. Finding simpler pivoting rules that achieve the same effect remains an open challenge. Nevertheless, we remark that for several important polytope classes---such as matching and flows---this edge-direction is easier to obtain: standard combinatorial algorithms often iteratively compute \emph{shortest} augmenting paths, which geometrically correspond to steepest edge-directions, and thus to the shadow edge used by the Shifting Shadow pivot-rule. Independently, we stress that our goal is theoretical more than practical: we aim at understanding the capabilities and limitations of the Simplex method, to contribute to the long-standing question of whether it can run in polynomial time---it is already known empirically to solve LPs quickly in practice (see, e.g., the discussion in \cite{BBHK26}).

Indeed, as a byproduct, we obtained an improved bound on the number of Simplex pivots required to escape degeneracy from a coordinate basis in 0/1-LPs, resolving an open question from~\cite{BDLKS24}.

We conclude by mentioning that finding an initial basis, a problem usually referred to as the \textit{Phase I}, is itself easy to model as a 0/1-LP if the original LP is 0/1. A well-known trick for this is to use a pyramid construction as in~\cite{Megiddo}. 
\subsection{Necessity of the Coordinate Basis}
We provide an example of a 0/1-LP of the form (\ref{lpi}) and a basis $B$ corresponding to one of its vertices $\vex'$ with the property that 
\begin{enumerate}
    \item the only available $\vecc$-improving pivot at $B$ is along an edge,
    \item the objective function improvement gained by pivoting along this edge is some arbitrarily small value $\varepsilon$, and
    \item the objective function improvement gained by pivoting along the steepest edge at $\vex'$ is $1+\varepsilon$.
\end{enumerate}
That is, it is possible to be at a basis $B$ corresponding to a vertex $\vex'$ of a 0/1-LP from which it is impossible to pivot along the steepest edge at $\vex'$ by performing a sequence of Simplex pivots (i.e., $\vecc$-increasing pivots).  Moreover, the only non-degenerate pivot which can be performed by performing Simplex pivots produces arbitrarily bad objective function improvement.  This demonstrates that the capacity for the Simplex method to follow particular paths (like that generated by steepest edges) or guarantee reasonable objective function improvement at each non-degenerate pivot is---in the presence of degeneracy---highly sensitive to \textit{which} basis corresponding to $\vex'$ it starts from.  This is true even in the setting of 0/1-LPs wherein the steepest edge-direction is guaranteed to improve the objective function by a minimum amount (relative to the current optimality gap).  This highlights both the value and necessity of pivoting to the coordinate basis  at each vertex.

Let $\varepsilon>0$ and consider the LP in two variables
\begin{align*}
		&\max\, \vex(1) + \varepsilon\vex(2) \\
		\text{s.t.}& \\
		&\vex(1)-\vex(2)\leq 0 \\
		 -2&\vex(1)-\varepsilon\vex(2)\leq 0 \\
        & \veo\leq\vex\leq \ve 1.
\end{align*}
It can be checked that the feasible region of this LP is the convex hull of the points $(0,0)$, $(0,1)$, and $(1,1)$.  Consider the basis $B$ corresponding to $(0,0)$ given by the inequalities $\vex(1)\geq 0$ and $-2\vex(1)-\varepsilon\vex(2)\leq 0$. The available pivot directions are $(0,1)$ and  $(\varepsilon,-2)$.

The only non-degenerate pivot direction is $(0,1)$, and this is also the only one of the available pivot directions with positive cost. Pivoting in this direction moves to the vertex $(0,1)$, resulting in an objective function improvement of $\varepsilon$.

On the other hand, the steepest edge at $(0,0)$ is the edge-direction $(1,1)$.  Pivoting in this direction moves to the vertex $(1,1)$, resulting in an objective function improvement of $1+\varepsilon$.

This construction makes use of the redundant inequality constraint $-2\vex(1)-\varepsilon\vex(2)\leq 0$.  Redundant inequalities appear frequently in LP formulations, especially for those modeling problems in combinatorial optimization which are themselves commonly 0/1-LPs.  Although identifying redundant equality constraints can be done efficiently, identifying redundant inequality constraints is no easier in general than solving the LP in question.  As such, one cannot ignore examples such as these coming from bases containing redundant inequalities.
\\\\
\noindent\textbf{AI Use Declaration:} No AI was used in any part of this work.

\bibliographystyle{elsarticle-num}
\bibliography{finalbib}

@article{BDLKS24,
  title={On the simplex method for 0/1-polytopes},
  author={Black, A. and De Loera, J. and Kafer, S. and Sanit{\`a}, L.},
  journal={Mathematics of Operations Research},
  year={2024},
  publisher={INFORMS}
}

@article{DKL2022,
author = {De Loera, J. and Kafer, S. and Sanit\`{a}, L.},
title = {Pivot Rules for Circuit-Augmentation Algorithms in Linear Optimization},
journal = {SIAM Journal on Optimization},
volume = {32},
number = {3},
pages = {2156-2179},
year = {2022},
doi = {10.1137/21M1419994},
URL = {  https://doi.org/10.1137/21M1419994 },
eprint = { https://doi.org/10.1137/21M1419994 },
}

@InProceedings{KS24,
author="Kukharenko, K.
and Sanit{\`a}, L.",
editor="Vygen, Jens
and Byrka, Jaros{\l}aw",
title="On the Number of Degenerate Simplex Pivots",
booktitle="Integer Programming and Combinatorial Optimization",
year="2024",
publisher="Springer Nature Switzerland",
address="Cham",
pages="252--264",
isbn="978-3-031-59835-7"
}

@book{GLS,
  author    = { Gr\"otschel ,M. and  Lov\'asz, L. and Schrijver, A.},
  title     = {Geometric Algorithms and Combinatorial Optimization},
  publisher = {Springer Berlin, Heidelberg},
  year      = {1993},
  ISBN ={978-3-642-78242-8},
}

@inproceedings{BBHK26,
  author    = {Bach, E. and Black, A. and 
               Huiberts, S. and Kafer, S.},
  title     = {Beyond Smoothed Analysis: Analyzing the Simplex Method By-the-Book},
  booktitle = {Proceedings of the 58th Annual {ACM} {SIGACT} Symposium on Theory
               of Computing, {STOC} 2026, Salt Lake City, UT, USA, June 22-27, 2026},
  pages     = {222--233},
  year      = {2026},
  doi = {https://doi.org/10.1145/3798129.380074},
}

@article{DantzigSimplex,
  title={Maximization of a linear function of variables subject to linear
inequalities},
  author={Dantzig, G.},
  journal={Koopmans, T.C. (ed.) Activity Analysis of Production and Allocation},
  pages={339--347},
  year={1951},
  publisher={John Wiley \& Sons Inc.}
}

@book{MINLP,
  title={Mixed integer nonlinear programming},
  author={Lee, J. and Leyffer, S.},
  volume={154},
  year={2011},
  publisher={Springer Science \& Business Media}
}

@book{Schrijver1986,
    AUTHOR = {Schrijver, A.},
     TITLE = {Theory of linear and integer programming},
    SERIES = {Wiley-Interscience Series in Discrete Mathematics},
      NOTE = {A Wiley-Interscience Publication},
 PUBLISHER = {John Wiley \& Sons Ltd.},
     PLACE = {Chichester},
      YEAR = {1986},
     PAGES = {xii+471},
      ISBN = {0-471-90854-1},
    REVIEW = {\MR{874114 (88m:90090)}},
}

@article{terlaky+zhang,
  author    = { Terlaky, T. and  Zhang, S.},
  title     = {Pivot rules for linear programming: {A} survey on recent theoretical
               developments},
  journal   = {Annals {OR}},
  volume    = {46-47},
  number    = {1},
  pages     = {203--233},
  year      = {1993},
  url       = {https://doi.org/10.1007/BF02096264},
  doi       = {10.1007/BF02096264}
}

@article{ShamirSurvey,
  title={The efficiency of the simplex method: a survey},
  author={Shamir, R.},
  journal={Management science},
  volume={33},
  number={3},
  pages={301--334},
  year={1987},
  publisher={INFORMS}
}

@article{Kar84,
title = {A new polynomial-time algorithm for linear programming},
author={N. Karmarkar},
journal = {Proceedings of STOC},
year = {1984}
}

@article{Kach79,
title = {A polynomial algorithm in linear programming},
author={L. Khachiyan},
journal = {Doklady Academii Nauk SSSR},
year = {1979}
}

@article{Smale00,
title = {Mathematical problems for the next century},
author={S. Smale},
journal = {Mathematics: frontiers and perspectives, American Mathematics Society},
year = {2000}
}

@article{Megiddo,
title = {A note on degeneracy in linear programming. Mathematical Programming},
author={N. Megiddo},
journal = {Mathematical Programming},
year = {1986}
}

@article{d4b74918-1611-3951-bc4e-72806d40b7e3,
 ISSN = {0003486X},
 URL = {http://www.jstor.org/stable/23234170},
 author = {Francisco Santos},
 journal = {Annals of Mathematics},
 number = {1},
 pages = {383--412},
 publisher = {Annals of Mathematics},
 title = {A counterexample to the {H}irsch Conjecture},
 urldate = {2025-03-31},
 volume = {176},
 year = {2012}
}

\end{document}